\documentclass[english]{amsart}
\usepackage[T1]{fontenc}
\usepackage[latin9]{inputenc}
\usepackage{amsthm}
\usepackage{amssymb}

\makeatletter
\numberwithin{equation}{section}
\numberwithin{figure}{section}
\theoremstyle{plain}
\newtheorem{thm}{\protect\theoremname}
\theoremstyle{plain}
\newtheorem{cor}[thm]{\protect\corollaryname}
\theoremstyle{remark}
\newtheorem{rem}[thm]{\protect\remarkname}

\makeatother

\usepackage{babel}
\providecommand{\corollaryname}{Corollary}
\providecommand{\remarkname}{Remark}
\providecommand{\theoremname}{Theorem}

\begin{document}

\title{On Positivities of Certain q-Special Functions}

\author{Ruiming Zhang}

\address{College of Science\\
Northwest A\&F University\\
Yangling, Shaanxi 712100\\
P. R. China.}

\email{ruimingzhang@yahoo.com}


\thanks{The work is supported by the National Natural Science Foundation
of China grants No. 11371294 and No. 11771355. }
\begin{abstract}
In this work we shall apply the Bochner's theorem to prove certain
combinations of Euler's q-exponentials are positive definite functions.
Then we apply this positivity to prove curious inequalities for the
Jacobi theta function $\vartheta_{4}$ and q-Gamma function $\Gamma_{q}$.
\end{abstract}

\maketitle

\section{Introduction}

For any $n\in\mathbb{N}$, the matrix $\left(a_{j,k}\right)_{j,k=1}^{n},\ a_{j,k}\in\mathbb{C}$
is called positive semidefinite if and only if ${\displaystyle \sum_{j,k=1}^{n}}a_{j,k}z_{j}\overline{z_{k}}\ge0$
for all $z_{1},\,z_{2},\dots\,z_{n}\in\mathbb{C}$. A positive semidefinite
matrix is positive definite if ${\displaystyle \sum_{j,k=1}^{n}}a_{j,k}z_{j}\overline{z_{k}}=0$
implies that $z_{1}=\dots=z_{n}=0$. Given two positive semidefinite
matrices 
\[
A=\left(a_{j,k}\right)_{j,k=1}^{n},\quad B=\left(b_{j,k}\right)_{j,k=1}^{n},\quad a_{j,k},\,b_{j,k}\in\mathbb{C},
\]
it is well-known that their Schur (Hadamard) product $A\circ B=\left(a_{j,k}b_{j,k}\right)_{j,k=1}^{n}$
is also positive semidefinite. Furthermore, it satisfies \cite{Horn,Zhang}
\[
\det\left(A\circ B\right)\ge\det(A)\cdot\det(B).
\]

A continuous function $f(x)$ on $\mathbb{R}$ with $f(0)=1$ is called
positive definite if and only if the matrices $\left(f(x_{j}-x_{k})\right)_{j,k=1}^{n}$
are positive semidefinite for all $n\in\mathbb{N}$ and $x_{1},\dots,x_{n}\in\mathbb{R}$.
The Bochner's theorem states that $f(x)$ is positive definite if
and only if there exists a probability measure $\mu(dx)$ on $\mathbb{R}$
such that \cite{Achieser}
\[
f(x)=\int_{\mathbb{R}}e^{ixy}\mu\left(dy\right).
\]
The condition $f(0)=1$ can be replaced by $f(0)>0$, then the total
mass of the measure $\mu(dx)$ is $f(0)$. If $f(x)$ is a positive
definite function, then $\overline{f(x)},\,\Re(f(x)),\,\left|f(x)\right|^{2},\,f(ax)$
are all positive definite functions where $a$ is any nonzero real
number. Furthermore, for any $m\in\mathbb{N}$ if $f_{1}(x),\dots,f_{m}(x)$
are positive definite functions, then $\prod_{n=1}^{m}f_{n}(x)$ and
$\sum_{n=1}^{m}a_{n}f_{n}(x)$ are also positive definite functions
where $a_{n}\ge0,\ n=1,\dots,m$ and $\sum_{n=1}^{m}a_{n}=1$.

In \cite{Ismail2} we derived numerous Fourier transformations for
several q-special functions. In this work we shall apply the Bochner's
theorem to prove certain combinations of Euler's q-exponentials are
positive definite functions. As a corollary of the positivities we
prove some curious inequalities for the Jacobi theta function $\vartheta_{4}$,
q-Gamma function $\Gamma_{q}$ and Euler's q-exponentials. 

\section{Main Results}
\begin{thm}
\label{thm:2.2}Let $n\in\mathbb{N}$. If $0<y_{\ell},q_{\ell}<1$
for all $1\le\ell\le n$, then 
\begin{equation}
\prod_{\ell=1}^{n}\frac{\exp\left(\frac{x^{2}}{\log q_{\ell}^{4}}\right)}{\left(y_{\ell}e^{ix};q_{\ell}\right)_{\infty}}\quad x\in\mathbb{R}\label{eq:2.1}
\end{equation}
 is a positive definite function. Furthermore, for all $m\in\mathbb{N}$
and all $x_{1},\dots,x_{m}\in\mathbb{R}$, the matrices 
\begin{equation}
\left(\prod_{\ell=1}^{n}\frac{\exp\left(\frac{\left(x_{j}-x_{k}\right)^{2}}{\log q_{\ell}^{4}}\right)}{\left(y_{\ell}e^{i(x_{j}-x_{k})};q_{\ell}\right)_{\infty}}\right)_{j,k=1}^{m}\label{eq:2.2}
\end{equation}
 are positive semidefinite. In particular, we have the inequality,
\begin{equation}
\prod_{\ell=1}^{n}\frac{\left(y_{\ell};q_{\ell}\right)_{\infty}}{\left|\left(y_{\ell}e^{ix};q_{\ell}\right)_{\infty}\right|}\exp\left(\frac{x^{2}}{\log q_{\ell}^{4}}\right)\le1,\quad x\in\mathbb{R}.\label{eq:2.3}
\end{equation}
\end{thm}

\begin{proof}
Since

\[
A_{q}(z)=\sum_{n=0}^{\infty}\frac{q^{n^{2}}}{(q;q)_{n}}(-z)^{n},
\]
then for $z>0,\ q\in(0,1)$ we have
\[
q^{\alpha^{2}}A_{q}\left(-q^{2\alpha}z\right)>0,\quad\alpha\in\mathbb{R}.
\]
Then by (5.37) of \cite{Ismail2}, 
\[
\frac{\exp\left(\frac{x^{2}}{\log q^{4}}\right)}{\left(ze^{ix};q\right)_{\infty}\sqrt{\log q^{-2}}}=\frac{1}{\sqrt{2\pi}}\int_{-\infty}^{\infty}q^{\alpha^{2}}A_{q}\left(-q^{2\alpha}z\right)\exp\left(-i\alpha x\right)d\alpha,
\]
where $|z|<1,\ x\in\mathbb{R}$, the inequality
\[
\frac{1}{\left(y;q\right)_{\infty}}>0,\quad y,q\in(0,1),
\]
and by the Bochner's theorem we see that the continuous function 
\begin{equation}
\frac{\exp\left(\frac{x^{2}}{\log q^{4}}\right)}{\left(ye^{ix};q\right)_{\infty}},\quad y,q\in(0,1)\label{eq:2.4}
\end{equation}
is positive semidefinite in variable $x\in\mathbb{R}$. Therefore,
for all $m\in\mathbb{N}$ and all distinct $x_{1},x_{2},\dots,x_{m}\in\mathbb{R}$,
the matrices
\begin{equation}
\left(\frac{\exp\left(\frac{(x_{i}-x_{j})^{2}}{\log q^{4}}\right)}{\left(ye^{i(x_{i}-x_{j})};q\right)_{\infty}}\right)_{i,j=1}^{m},\quad q,y\in(0,1)\label{eq:2.5}
\end{equation}
are positive semidefinite. (\ref{eq:2.1}) and (\ref{eq:2.2}) are
obtained by taking the Schur (Hadamard) product of (\ref{eq:2.4})
or (\ref{eq:2.5}) respectively. For all $m\in\mathbb{N}$, by (\ref{eq:2.2})
we get
\[
\det\left(\prod_{\ell=1}^{n}\frac{\exp\left(\frac{\left(x_{j}-x_{k}\right)^{2}}{\log q_{\ell}^{4}}\right)}{\left(y_{\ell}e^{i(x_{i}-x_{j})};q_{\ell}\right)_{\infty}}\right)_{j,k=1}^{m}>0.
\]
Then (\ref{eq:2.3}) is obtained by setting $m=2$ in the above inequality.
\end{proof}
The $q$-Gamma function is defined by \cite{Ismail2}

\[
\Gamma_{q}(x)=(1-q)^{1-x}\frac{(q;q)_{\infty}}{(q^{x};q)_{\infty}},
\]
and the Jacobi theta function $\vartheta_{4}$ is \cite{Ismail2}

\[
\vartheta_{4}\left(z;q\right)=\vartheta_{4}\left(v\vert\tau\right)=\sum_{n=-\infty}^{\infty}q^{n^{2}}\left(-z\right)^{n}=\left(q^{2},qz,q/z;q^{2}\right)_{\infty},
\]
where
\[
z=e^{2\pi iv},\ q=e^{\pi i\tau},\quad\Im(\tau)>0.
\]

\begin{cor}
If $0<q<1$, $u>0$ and $v\in\mathbb{R}$, then
\begin{equation}
\left|\left(q^{u+iv};q\right)_{\infty}\right|\ge\left(q^{u};q\right)_{\infty}q^{v^{2}/4},\label{eq:2.6}
\end{equation}
\begin{equation}
\Gamma_{q}(u)\ge\left|\Gamma_{q}(u+iv)\right|q^{v^{2}/4}\label{eq:2.7}
\end{equation}
and 
\begin{equation}
\frac{\vartheta_{4}\left(v\vert iu\right)}{\vartheta_{4}\left(0\vert iu\right)}\ge e^{-\pi v^{2}/u}.\label{eq:2.8}
\end{equation}
 
\end{cor}

\begin{proof}
From (\ref{eq:2.3}) we get
\begin{equation}
\left|\left(ye^{ix};q\right)_{\infty}\right|\ge\left(y;q\right)_{\infty}\exp\left(\frac{x^{2}}{4\log q}\right),\quad0<y,q<1,\ x\in\mathbb{R}.\label{eq:2.9}
\end{equation}
Let 
\[
x=v\log q,\ y=q^{u},\quad u\in(0,\infty),\ v\in\mathbb{R}
\]
in (\ref{eq:2.9}) we get
\[
\left|\left(q^{u+iv};q\right)_{\infty}\right|\ge\left(q^{u};q\right)_{\infty}q^{v^{2}/4}
\]
and
\[
\Gamma_{q}(u)\ge\left|\Gamma_{q}(u+iv)\right|q^{v^{2}/4},\quad0<q<1,\ u>0,\ v\in\mathbb{R}.
\]
Multiply $(q;q)_{\infty}$ to the square of (\ref{eq:2.9}) to obtain,
\[
\left(q,ye^{ix},ye^{-ix};q\right)_{\infty}\ge\left(q,y,y;q\right)_{\infty}\exp\left(\frac{x^{2}}{2\log q}\right).
\]
 Then let $q\to q^{2},\ y\to q$ in the above inequality to obtain
\begin{equation}
\left(q^{2},qe^{ix},qe^{-ix};q^{2}\right)_{\infty}\ge\left(q^{2},q,q;q^{2}\right)_{\infty}\exp\left(\frac{x^{2}}{4\log q}\right),\label{eq:2.10}
\end{equation}
which gives 
\[
\vartheta_{4}(v\vert ui)\ge\vartheta_{4}(0\vert ui)e^{-\pi v^{2}/u}
\]
 by setting 
\[
q=e^{-\pi u},\ x=2\pi v,\quad u>0,\ v\in\mathbb{R}.
\]
\end{proof}
\begin{rem}
Given any simplify connected region $\Omega$ that is not the entire
complex plane, then by Riemann's mapping theorem, \cite{Ahlfors},
there exists an analytic function $f(z)=u(z)+iv(z)$ in $\Omega$
defines a one-to-one mapping of $\Omega$ onto the right half-plane
$H=\left\{ u+iv|u>0,\ v\in\mathbb{R}\right\} $. The inequalities
(\ref{eq:2.6}), (\ref{eq:2.7}) and (\ref{eq:2.8}) imply that
\begin{equation}
\left|\left(q^{f(z)};q\right)_{\infty}\right|\ge\left(q^{\Re(f(z))};q\right)_{\infty}q^{(\Im(z))^{2}/4},\quad z\in\Omega\label{eq:2.11}
\end{equation}
\begin{equation}
\Gamma_{q}\left(\Re(f(z))\right)\ge\left|\Gamma_{q}\left(f(z)\right)\right|q^{(\Im(z))^{2}/4},\quad z\in\Omega.\label{eq:2.12}
\end{equation}
By \cite{Rademacher}
\[
\vartheta_{4}\left(v+\tau\vert\tau\right)=-e^{-\pi i\tau-2\pi iv}\vartheta_{4}\left(v\vert\tau\right),\quad\Im(\tau)>0,
\]
for $u>0,\ v\in\mathbb{R}$ we get
\[
\vartheta_{4}\left(ui-v\vert ui\right)=-e^{\pi u+2\pi iv}\vartheta_{4}\left(-v\vert ui\right)=-e^{\pi u+2\pi iv}\vartheta_{4}\left(v\vert ui\right),
\]
hence,
\begin{equation}
\left|\vartheta_{4}\left(ui-v\vert ui\right)\right|\ge\vartheta_{4}(0\vert ui)e^{\pi u\left(1-v^{2}/u^{2}\right)}.\label{eq:2.13a}
\end{equation}
Then for $z\in\Omega$,
\begin{equation}
\left|\vartheta_{4}\left(if(z)\vert\Re(f(z))i\right)\right|\ge\vartheta_{4}\left(0\vert\Re(f(z))i\right)\exp\pi\left(\Re(f(z))-\frac{(\Im(z))^{2}}{\Re(f(z))}\right).\label{eq:2.13b}
\end{equation}
In particular, since the function
\begin{equation}
f(z)=\frac{1+z}{1-z}=\frac{1-x^{2}-y^{2}}{(1-x)^{2}+y^{2}}+\frac{2yi}{(1-x)^{2}+y^{2}}\label{eq:2.14}
\end{equation}
 maps unit disk unto the right-half plane. Then for $z=x+iy$ with
$x^{2}+y^{2}<1$ we have
\begin{equation}
\left|\left(q^{\frac{1+z}{1-z}};q\right)_{\infty}\right|\ge\left(q^{\frac{1-x^{2}-y^{2}}{(1-x)^{2}+y^{2}}};q\right)_{\infty}q^{\frac{y^{2}}{\left((1-x)^{2}+y^{2}\right)^{2}}},\label{eq:2.15}
\end{equation}
 
\begin{equation}
\Gamma_{q}\left(\frac{1-x^{2}-y^{2}}{(1-x)^{2}+y^{2}}\right)\ge\left|\Gamma_{q}\left(\frac{1+z}{1-z}\right)\right|q^{\frac{y^{2}}{\left((1-x)^{2}+y^{2}\right)^{2}}}\label{eq:2.16}
\end{equation}
and
\begin{align}
 & \left|\vartheta_{4}\left(i\frac{1+z}{1-z}\bigg|\frac{1-x^{2}-y^{2}}{(1-x)^{2}+y^{2}}i\right)\right|\ge\vartheta_{4}\left(0\bigg|\frac{1-x^{2}-y^{2}}{(1-x)^{2}+y^{2}}i\right)\label{eq:2.17}\\
 & \times\exp\left(\frac{\pi\left(4y^{2}-\left(1-x^{2}-y^{2}\right)^{2}\right)}{\left((1-x)^{2}+y^{2}\right)\left(1-x^{2}-y^{2}\right)}\right).\nonumber 
\end{align}
 
\end{rem}

\begin{thm}
Let $n\in\mathbb{N}$. If $0<z_{\ell},b_{\ell},q_{\ell}<1,\ -1<a_{\ell}<1,$
for all $1\le\ell\le n$, then the function 
\begin{equation}
\prod_{\ell=1}^{n}\frac{\left(a_{\ell}z_{\ell}e^{ix};q_{\ell}\right)_{\infty}\exp\left(\frac{x^{2}}{\log q_{\ell}^{2}}\right)}{\left(b_{\ell}e^{ix},-z_{\ell}e^{ix};q_{\ell}\right)_{\infty}}\label{eq:2.18}
\end{equation}
is positive definite in $x$. In particular, for all $m\in\mathbb{N}$
and distinct $x_{1},\dots,x_{m}\in\mathbb{R}$ the matrices 
\begin{equation}
\left(\prod_{\ell=1}^{n}\frac{\left(a_{\ell}z_{\ell}e^{i(x_{j}-x_{k})};q_{\ell}\right)_{\infty}\exp\left(\frac{(x_{j}-x_{k})^{2}}{\log q_{\ell}^{2}}\right)}{\left(b_{\ell}e^{i(x_{j}-x_{k})},-z_{\ell}e^{i(x_{j}-x_{k})};q_{\ell}\right)_{\infty}}\right)_{j,k=1}^{m}\label{eq:2.19}
\end{equation}
are positive semidefinite. Furthermore, for $x\in\mathbb{R}$ we have
\begin{equation}
\prod_{\ell=1}^{n}\left|\frac{\left(b_{\ell},-z_{\ell};q\right)_{\infty}}{\left(b_{\ell}e^{ix},-z_{\ell}e^{ix};q\right)_{\infty}}\cdot\frac{\left(a_{\ell}z_{\ell}e^{ix};q_{\ell}\right)_{\infty}}{\left(a_{\ell}z_{\ell};q_{\ell}\right)_{\infty}}\right|\le\prod_{\ell=1}^{n}\exp\left(\frac{x^{2}}{\log q_{\ell}^{2}}\right).\label{eq:2.20}
\end{equation}
\end{thm}

\begin{proof}
For $q\in(0,1)$, the confluent basic hypergeometric series 
\[
_{1}\phi_{1}\left(a;b;q,z\right)=\sum_{n=0}^{\infty}\frac{\left(a;q\right)_{n}q^{\binom{n}{2}}}{\left(b,q;q\right)_{n}}\left(-z\right)^{n}
\]
satisfies 
\[
_{1}\phi_{1}\left(a;b;q,-z\right)>0,\quad z>0,\ a,b\in(-\infty,1).
\]
By (5.83) in \cite{Ismail2},
\[
\frac{\left(-aze^{ix};q\right)_{\infty}\exp\left(\frac{x^{2}}{\log q^{2}}\right)}{\left(-bq^{-1/2}e^{ix},ze^{ix};q\right)_{\infty}\sqrt{\log q^{-1}}}=\sqrt{\frac{\log q^{-1}}{2\pi}}\int_{-\infty}^{\infty}\left(bq^{\alpha};q\right)_{\infty}q^{\alpha^{2}/2}{}_{1}\phi_{1}\left(a;bq^{\alpha};q,zq^{\alpha+1/2}\right)\frac{e^{-i\alpha x}d\alpha}{\sqrt{2\pi}}
\]
we get
\begin{align*}
 & \frac{\left(aze^{ix};q\right)_{\infty}\exp\left(\frac{x^{2}}{\log q^{2}}\right)}{\left(be^{ix},-ze^{ix};q\right)_{\infty}\sqrt{\log q^{-1}}}\\
 & =\int_{-\infty}^{\infty}\left(-bq^{\alpha+\frac{1}{2}};q\right)_{\infty}q^{\alpha^{2}/2}{}_{1}\phi_{1}\left(a;-bq^{\alpha+\frac{1}{2}};q,-zq^{\alpha+\frac{1}{2}}\right)\frac{e^{-i\alpha x}d\alpha}{\sqrt{2\pi}}.
\end{align*}
Since for $q,b,z\in(0,1),\ a\in(-1,1),\ \alpha\in\mathbb{R}$ we have
\[
\left(-bq^{\alpha+\frac{1}{2}};q\right)_{\infty}q^{\alpha^{2}/2}{}_{1}\phi_{1}\left(a;-bq^{\alpha+\frac{1}{2}};q,-zq^{\alpha+\frac{1}{2}}\right)>0
\]
and
\[
\frac{\left(az;q\right)_{\infty}}{\left(b,-z;q\right)_{\infty}\sqrt{\log q^{-1}}}>0.
\]
Then by Bochner's theorem we know that the continuous function,
\begin{equation}
\frac{\left(aze^{ix};q\right)_{\infty}\exp\left(\frac{x^{2}}{\log q^{2}}\right)}{\left(be^{ix},-ze^{ix};q\right)_{\infty}}\label{eq:2.21}
\end{equation}
is a positive definite function in $x\in\mathbb{R}$. (\ref{eq:2.18})
and (\ref{eq:2.19}) are obtained by taking Schur (Hadamard) products,
and (\ref{eq:2.20}) comes from the determinant of (\ref{eq:2.19})
at $m=2$.
\end{proof}
\begin{cor}
If $0<a,b,q<1$, $-1<c<1$ and $v\in\mathbb{R}$, then
\begin{equation}
\left|\frac{\left(aq^{iv},-bq^{iv};q\right)_{\infty}}{\left(bce^{ix};q\right)_{\infty}}\right|\ge\frac{\left(a,-b;q\right)_{\infty}}{\left(bc;q\right)_{\infty}}q^{v^{2}/2}.\label{eq:2.23}
\end{equation}
 In particular,
\begin{equation}
\left|\frac{\left(b^{2}q^{2iv};q^{2}\right)_{\infty}}{\left(abq^{iv};q\right)_{\infty}}\right|\ge\frac{\left(b^{2};q^{2}\right)_{\infty}}{\left(ab;q\right)_{\infty}}q^{v^{2}/2}.\label{eq:2.24}
\end{equation}
\end{cor}

\begin{proof}
For $0<z,b,q<1$, $-1<a<1$ and $x\in\mathbb{R}$, let $n=1$ in (\ref{eq:2.20})
we get 

\[
\left|\frac{\left(be^{ix},-ze^{ix};q\right)_{\infty}}{\left(aze^{ix};q\right)_{\infty}}\right|\ge\frac{\left(b,-z;q\right)_{\infty}}{\left(az;q\right)_{\infty}}\exp\left(\frac{x^{2}}{\log q^{2}}\right).
\]
 In particular, 
\[
\left|\frac{\left(b^{2}e^{2ix};q^{2}\right)_{\infty}}{\left(abe^{ix};q\right)_{\infty}}\right|\ge\frac{\left(b^{2};q^{2}\right)_{\infty}}{\left(ab;q\right)_{\infty}}\exp\left(\frac{x^{2}}{\log q^{2}}\right).
\]
Then set $x=v\log q,\ v\in\mathbb{R}$ to get
\[
\left|\frac{\left(bq^{iv},-zq^{iv};q\right)_{\infty}}{\left(aze^{ix};q\right)_{\infty}}\right|\ge\frac{\left(b,-z;q\right)_{\infty}}{\left(az;q\right)_{\infty}}q^{v^{2}/2}
\]
 and
\[
\left|\frac{\left(b^{2}q^{2iv};q^{2}\right)_{\infty}}{\left(abq^{iv};q\right)_{\infty}}\right|\ge\frac{\left(b^{2};q^{2}\right)_{\infty}}{\left(ab;q\right)_{\infty}}q^{v^{2}/2}.
\]
\end{proof}
Then the corollary is obtained by renaming variables.

Let $q=e^{-2k^{2}}$ and $|q|<1$, from a Ramanujan's identity \cite{Gasper}
\begin{align*}
 & \int_{-\infty}^{\infty}\frac{e^{-x^{2}+2mx}dx}{\left(ae^{2ikx}q^{1/2},be^{-2ikx}q^{1/2};q\right)_{\infty}}\\
 & =\frac{\sqrt{\pi}e^{m^{2}}\left(-aqe^{2imk},-bqe^{-2imk};q\right)_{\infty}}{\left(abq;q\right)_{\infty}}
\end{align*}
 we get 
\begin{align*}
 & \int_{-\infty}^{\infty}\frac{e^{-x^{2}+2imx}dx}{\left|\left(ce^{2ikx};q\right)_{\infty}\right|^{2}}\\
 & =\frac{\sqrt{\pi}e^{-m^{2}}\left(-cq^{1/2}e^{-2mk},-\overline{c}q^{1/2}e^{2mk};q\right)_{\infty}}{\left(\left|c\right|^{2};q\right)_{\infty}}.
\end{align*}
 Then we have the following:
\begin{thm}
For all $n\in\mathbb{N}$ and 
\begin{equation}
0<c_{j}<1,\ k_{j}>0,\ q_{j}=e^{-2k_{j}^{2}},\ x\in\mathbb{R},\label{eq:2.25}
\end{equation}
the function 
\begin{equation}
\prod_{j=1}^{n}e^{-x^{2}}\left(-c_{j}e^{-k_{j}^{2}-2xk_{j}},-c_{j}e^{-k_{j}^{2}+2xk_{j}};q_{j}\right)_{\infty}\label{eq:2.26}
\end{equation}
is positive definite. Furthermore, for all $m\in\mathbb{N}$ and $x_{1},\dots,x_{m}\in\mathbb{R}$
the matrices
\begin{equation}
\left(\prod_{j=1}^{n}e^{-\left(x_{r}-x_{s}\right)^{2}}\left(-c_{j}e^{-k_{j}^{2}-2\left(x_{r}-x_{s}\right)k_{j}},-c_{j}e^{-k_{j}^{2}-2\left(x_{s}-x_{r}\right)k_{j}};q_{j}\right)_{\infty}\right)_{r,s=1}^{m}\label{eq:2.27}
\end{equation}
 are positive semidefinite and for $0<c<q^{1/2},\ k>0,\ q=e^{-2k^{2}}$
we have
\begin{equation}
\left(-ce^{-2xk},-ce^{2xk};q\right)_{\infty}\le e^{x^{2}}\left(-c,-c;q\right)_{\infty},\quad x\in\mathbb{R}.\label{eq:2.28}
\end{equation}
 
\end{thm}


\begin{thebibliography}{1}
\bibitem{Achieser} N. I. Achieser, Theory of Approximation, Dover
Publications, Inc., New York, 1992.

\bibitem{Ahlfors}L. Ahlfors, Complex Analysis, 3rd edition, McGraw-Hill,
1979.

\bibitem{Gasper} G. Gasper and M. Rahman, \emph{Basic Hypergeometric
Series}, Cambridge University Press, 2nd edition, Cambridge, 2004.

\bibitem{Horn} R. A. Horn and C. R. Johnson, Matrix Analysis, Cambridge
University Press, Cambridge 1992.

\bibitem{Ismail1} R. A. Horn and C. R. Johnson, Matrix Analysis,
Cambridge University Press, Cambridge 1992.M. E. H. Ismail, \emph{Classical
and Quantum Orthogonal Polynomials in One Variable, }Cambridge University
Press, Cambridge, 2005.

\bibitem{Ismail2}Mourad Ismail and R. Zhang, Integral and series
representations of q-polynomials and functions: Part I, Analysis and
Applications, Volume 16, (2), 209-281. 

\bibitem{Rademacher}H. Rademacher, Topics in Analytic Number Theory,
Die Grundlehren der math. Wissenschaften, Band 169, Springer-Verlag,
Berlin, 1973.

\bibitem{Zhang}R. Zhang, On Certain Positive Semidefinite Matrices
of Special Functions, chapter 27 in \textquotedblleft Frontiers of
Orthogonal Polynomials and q-Series\textquotedblright , edited by
X. Li and Z. Nashed, World Scientific, 2018.
\end{thebibliography}
\end{document}